\documentclass{amsart}

\usepackage{amsmath}
\usepackage{amssymb}
\usepackage{indentfirst}
\usepackage{galois}
\usepackage{amsthm}
\usepackage{amsfonts}

\newtheorem{theorem}{Theorem}[section]
\newtheorem{proposition}[theorem]{Proposition}
\newtheorem{lemma}[theorem]{Lemma}
\newtheorem{corollary}[theorem]{Corollary}

\theoremstyle{definition}

\theoremstyle{remark}
\newtheorem{remark}[theorem]{Remark}

\begin{document}

\bibliographystyle{plain}

\title[Spectra of some invertible weighted composition operators]{Spectra of some invertible weighted composition operators on Hardy and weighted Bergman spaces in the unit ball}

\author[Y.X. Gao and Z.H. Zhou] {Yong-Xin Gao and Ze-Hua Zhou$^*$}
\address{\newline  Yong-Xin Gao, Department of Mathematics,
Tianjin University, Tianjin 300072, P.R. China.}

\email{tqlgao@163.com}

\address{\newline Ze-Hua Zhou\newline Department of Mathematics, Tianjin University, Tianjin 300072,
P.R. China. \newline
Center for Applied Mathematics, Tianjin University, Tianjin 300072,
P.R. China.}
\email{zehuazhoumath@aliyun.com;zhzhou@tju.edu.cn}


\date{}
\thanks{\noindent $^*$Corresponding author.\\
This work was supported in part by the National Natural Science
Foundation of China (Grant Nos. 11371276; 11301373; 11201331).}

\begin{abstract}
In this paper, we investigate the spectra of invertible weighted composition operators with automorphism symbols, on Hardy space $H^2(\mathbb{B}_N)$ and weighted Bergman spaces $A_\alpha^2(\mathbb{B}_N)$, where $\mathbb{B}_N$ is the unit ball of the
$N$-dimensional complex space.  By taking $N=1$, $\mathbb{B}_N=\mathbb{D}$ the unit disc, we also complete the discussion about the spectrum of a weighted composition operator when it is invertible on $H^2(\mathbb{D})$ or $A_\alpha^2(\mathbb{D})$.
\end{abstract}
\maketitle

\section{Introduction}

Let $\mathbb{B}_N$ be the unit ball in $\mathbb{C}^N$ and $S_N$ denote the unit sphere. Let $H(\mathbb{B}_N)$ be the space of all holomorphic functions
on $\mathbb{B}_N$. The Hardy space $H^p(\mathbb{B}_N)$ is the set of holomorphic functions on $\mathbb{B}_N$ such that
$$||f||_{H^p}^p=\sup_{0<r<1}\int_{S_N}|f(r\zeta)|^pd\sigma(\zeta)<\infty,$$
where $d\sigma$ is the normalized surface measure on $S_N$. For $\alpha>-1$, the weighted Bergman space $A_\alpha^p(\mathbb{B}_N)$ is the set of holomorphic functions on $\mathbb{B}_N$ such that
$$||f||_{A_\alpha^p}^p=\int_{\mathbb{B}_N}c_\alpha|f(z)|^p(1-|z|^2)^\alpha dv(z)<\infty,$$
where $dv$ is the volume measure on $\mathbb{B}_N$ and $c_\alpha$ is a positive constant chosen so that
$$dv_\alpha(z)=c_\alpha(1-|z|^2)^\alpha dv(z)$$
is normalized. When taking $p=2$, the spaces are Hilbert.

Let $\varphi$ be a holomorphic map from $\mathbb{B}_N$ into itself and $\psi$ be a holomorphic function on $\mathbb{B}_N$. Then we define a weighted composition operator on $H(\mathbb{B}_N)$, by
$$C_{\psi,\varphi}f=\psi\cdot f\comp\varphi,$$
which is a linear operator regarded as generalization of a multiplication operator $M_{\psi}$ by putting $\varphi=Id$ the
identity map of $\mathbb{B}_N$, and a composition operator $C_{\varphi}$ when taking $\psi=1$.

Weighted composition operators arise naturally in the study of many subjects. For example, Forelli \cite{For} showed that all surjective isometries of the Hardy space $H^p(\mathbb{D})$ are weighted composition operators when the space is not Hilbert. A similar result on weighted Bergman space $A^p_{\alpha}(\mathbb{D})$ is due to Kolaski \cite{Kol}. It has also been shown that the commutants of analytic Toeplitz operators whose symbols are finite Blaschke products are exactly the multiple valued weighted composition operators, which is defined in \cite{CG}.

Over the past fifty years, composition operators have been actively investigated from a various points of view. Such as the boundness, compactness, spectrum and hypercyclicity, as well as the topological structure of the set of composition operators on various analytic function spaces. The recent papers or books \cite{Bay1,Bay2,BC,BLW,CM,HO1,HO2,M,Sha,SZ,ZC,ZS,ZZ} and the related references therein are good sources for information on much of the developments in the theory of (weighted) composition operators.

The spectral properties of weighted composition operators on the spaces of analytic functions on the unit disc $\mathbb{D}$ have also been discussed in many papers. Most of them are originated from the work of Kamowitz in \cite{Kam1,Kam2}, which is still instructive until now. For a non-automorphic symbol $\varphi$ with a fixed point in $\mathbb{D}$, Aron and Lindstr\"{o}m \cite{AL} completely described the spectrum of a weighted composition operator $C_{\psi,\varphi}$ acting on the weighted Banach
spaces of $H^{\infty}$-type.

In 2011, Gajath Gunatillake \cite{Gun2} showed that a weighted composition operator $C_{\psi,\varphi}$ is invertible on $H^2(\mathbb{D})$ if and only if both $\psi$ and $\frac{1}{\psi}$ are bounded  on $\mathbb{D}$ and $\varphi$ is an automorphism of $\mathbb{D}$. Then he investigated the spectrum of  $C_{\psi,\varphi}$ on the space $H^2(\mathbb{D})$ when it is invertible, with an extra hypothesis of the continuity of $\psi$ on $\overline{\mathbb{D}}$. He got a complete result when $\varphi$ is either elliptic or parabolic, see Theorem 3.1.1, Theorem 3.2.1 and Theorem 3.3.1 in \cite{Gun2}.

In 2013, Hyv\"{a}rinen et al \cite{Hyv} generalized Gunatillake's work onto more spaces, such as $H^p(\mathbb{D})$, $A_{\alpha}^p(\mathbb{D})$ and $H_{v_p}^{\infty}(\mathbb{D})$. Moreover, for a hyperbolic automorphism $\varphi$, they get the spectral radius of $C_{\psi,\varphi}$. Then they find out the spectrum of $C_{\psi,\varphi}$ in the case when $|\psi(a)/\varphi'(a)^s|$ is no less than $|\psi(b)/\varphi'(b)^s|$, where $a$ is the attractive fixed point of $\varphi$, $b$ is the repulsive fixed point of $\varphi$ and $s$ depends on the space. See theorem 4.9 in \cite{Hyv}.

The thing remains is the case when $|\psi(a)/\varphi'(a)^s|$ is large than $|\psi(b)/\varphi'(b)^s|$. In this paper, we continue the discussion in \cite{Gun2,Hyv} and give the result for this case, as a corollary of our main results. In fact, the spaces we consider in this paper consist of functions defined on $\mathbb{B}_N$ instead of the unit disc $\mathbb{D}$, such as $H^2(\mathbb{B}_N)$ and $A_\alpha^2(\mathbb{B}_N)$. Then we determine the spectrum of invertible $C_{\psi,\varphi}$ when $\varphi$ is an automorphism of $\mathbb{B}_N$ with no fixed point in $\mathbb{B}_N$, which are Theorem 3.11 and Theorem 4.7.

Finally, we get Corollary 3.12 as a special case of Theorem 3.11, which give a complete result about the spectra of invertible $C_{\psi,\varphi}$ on $H^2(\mathbb{D})$ and $A_\alpha^2(\mathbb{D})$ when $\varphi$ is hyperbolic.

\section{Preliminaries}

\subsection{Spaces}

The spaces we consider throughout the paper are $H^p(\mathbb{B}_N)$ and $A_\alpha^p(\mathbb{B}_N)$ when p=2, or equivalently, when they are Hilbert spaces.

Recall that the reproducing kernel or the point evaluation kernel at $w$ of $H^2(\mathbb{B}_N)$ is
$$\left(\frac{1}{1-\langle z,w\rangle}\right)^N,$$
and its norm is $\left(1-|w|^2\right)^{-N/2}$.
Also, in $A_\alpha^2(\mathbb{B}_N)$ the kernel for evaluation at $w$ is
$$\left(\frac{1}{1-\langle z,w\rangle}\right)^{N+1+\alpha},$$
and its norm is $\left(1-|w|^2\right)^{-(N+1+\alpha)/2}$.

Let $\{\beta(n)\}_{n=0}^\infty$ be a sequence of positive numbers. Then we define $H^2(\beta,\mathbb{B}_N)$, if possible, as a Hilbert space, in which the monomials form a complete orthogonal set, consisting of holomorphic functions on $\mathbb{B}_N$ such that
$$||f||^2=\sum_{s=0}^{\infty}\beta(s)^2||f_s||_{H^2}^2<\infty,$$
where $f=\sum_{s=0}^{\infty}f_s$ is the homogeneous expansion of the function.

We are particularly interested in the cases when $\beta(n)^2=(1+n)^{1-\gamma}$ with $\gamma\geqslant1$. When $\gamma=1$, the space $H^2(\beta,\mathbb{B}_N)$ is exactly $H^2(\mathbb{B}_N)$. When $\gamma>1$, it is not hard to check by Stirling's formula that
$H^2(\beta,\mathbb{B}_N)$ and $A_{\gamma-2}^2(\mathbb{B}_N)$ consist of same functions, and the norms on the two spaces are equivalent.

For more details about $H^2(\beta,\mathbb{B}_N)$, one can turn to Section 2.1 in \cite{CM}.

\begin{remark}
By the equivalence of the norms, if $\beta(n)=(n+1)^{1-\gamma}$ with $\gamma>1$, then the spectra of a weighted composition on $H^2(\beta,\mathbb{B}_N)$ and $A_{\gamma-2}^2(\mathbb{B}_N)$ are same. Thus, we may as well equip $H^2(\beta,\mathbb{B}_N)$ with the norm on $A_{\gamma-2}^2(\mathbb{B}_N)$. Hence we can focus on the spaces $H^2(\beta,\mathbb{B}_N)$ for $\beta(n)=(n+1)^{1-\gamma}$ with $\gamma\geqslant1$, instead of $H^2(\mathbb{B}_N)$ and $A_\alpha^2(\mathbb{B}_N)$. Straightforwardly, throughout this paper, by writing $H^2(\beta,\mathbb{B}_N)$ for $\beta(n)=(n+1)^{1-\gamma}$ we just mean $H^2(\mathbb{B}_N)$ when $\gamma=1$ and $A_{\gamma-2}^2(\mathbb{B}_N)$ when $\gamma>1$.
\end{remark}
Let $K_w$ be the point evaluation kernel at $w$ of $H^2(\beta,\mathbb{B}_N)$, then according to Remark 2.1,
$$K_w(z)=\left(\frac{1}{1-\langle z,w\rangle}\right)^{2K},$$
where $K=(N+1-\gamma)/2$. Also we have
$$||K_w||=\left(\frac{1}{1-|w|^2}\right)^K.$$

\subsection{Automorphisms of $\mathbb{B}_N$}

Let $\varphi$ be an automorphism of $\mathbb{B}_N$. If $\varphi$ has no fixed point in $\mathbb{B}_N$, then by Proposition 2.2.9 in \cite{Aba}, $\varphi$ has at least one and at most two fixed points on $S_N$. Moreover, Theorem 2.83 in \cite{CM} shows that $\varphi_k$ converge to one of the fixed points uniformly on compact subsets of $\mathbb{B}_N$. We call it the Denjoy-Wolff point of $\varphi$. Thus an automorphism $\varphi$ of $\mathbb{B}_N$ with no fixed point in $\mathbb{B}_N$ falls into one of the two classes below:

$\varphi$ fixes two distinct points on $S_N$. Then one of the fixed points is the Denjoy-Wolff point of $\varphi$.

$\varphi$ fixes one point on $S_N$. Then the only fixed point is the Denjoy-Wolff point of $\varphi$.

Note that the two classes above are exactly the generalization of hyperbolic and parabolic automorphisms of the unit disk $\mathbb{D}$ respectively. We will treat the two cases separately in this paper.

For any automorphism $\varphi$ of $\mathbb{B}_N$ we have
$$1-|\varphi(z)|^2=\frac{(1-|z|^2)(1-|a|^2)}{|1-\langle z,a\rangle|^2}.$$
This equation will be used repeatedly.

For any two points $z$ and $w$ in $\mathbb{B}_N$, define $d(z,w)=|\varphi_w(z)|$, where $\varphi_w$ is the the involutive automorphism that exchange $0$ and $a$. Then $d(\cdot,\cdot)$ gives a metric on $\mathbb{B}_N$. We call it the pseudo-hyperbolic metric of $\mathbb{B}_N$ and $d(z,w)$ is the pseudo-hyperbolic distance between $z$ and $w$. It is easy to check that the pseudo-hyperbolic distance is invariant under automorphisms of $\mathbb{B}_N$, that is,
$$d(\varphi(z),\varphi(w))=d(z,w)$$
whenever $\varphi$ is an automorphism of $\mathbb{B}_N$. Moreover, we have
$$1-d(z,w)^2=\frac{(1-|z|^2)(1-|w|^2)}{|1-\langle z,w\rangle|^2}.$$

Let $\varphi$ be a holomorphic map from $\mathbb{B}_N$ into itself. We use $\varphi_k$ denote the $k$-th iterate of $\varphi$ for $k\in\mathbb{N}^*$. If $\varphi$ is an automorphism of $\mathbb{B}_N$, $\varphi$ is invertible. Then we define $\varphi_k$ as the $|k|$-th iterate of $\varphi^{-1}$ for $k<0$. When $k=0$, we set $\varphi_0$ be the identity function of $\mathbb{B}_N$.

We use $Aut(\mathbb{B}_N)$ to denote the set of all automorphisms of $\mathbb{B}_N$ throughout the paper.

\subsection{Weighted composition operators}

Here we list some facts about the weighted composition operators. They are the fundamental of our discussion. All the facts can be checked easily.

\begin{proposition}
Suppose $C_{\psi_1,\varphi_1}$ and $C_{\psi_2,\varphi_2}$ are bounded on $H^2(\beta,\mathbb{B}_N)$, then $C_{\psi_1,\varphi_1}C_{\psi_2,\varphi_2}$ is also a weighted composition operator on $H^2(\beta,\mathbb{B}_N)$.
\end{proposition}

\begin{proof}
For any $f\in H^2(\beta,\mathbb{B}_N)$,
\begin{align*}
C_{\psi_1,\varphi_1}C_{\psi_2,\varphi_2}f&=C_{\psi_1,\varphi_1}(\psi_2\cdot f\comp\varphi_2)
\\&=\psi_1\cdot\psi_2\comp\varphi_1\cdot f\comp\varphi_2\comp\varphi_1.
\end{align*}
Thus we have
$$C_{\psi_1,\varphi_1}C_{\psi_2,\varphi_2}=C_{(\psi_1\cdot\psi_2\comp\varphi_1),(\varphi_2\comp\varphi_1)}.$$
\end{proof}

Let $\varphi$ be a holomorphic map from $\mathbb{B}_N$ into itself and $\psi$ be a holomorphic function on $\mathbb{B}_N$. We define
$$\psi_{(k)}=\prod_{j=0}^{k-1}\psi\comp\varphi_j$$
for $k\in\mathbb{N}^*$. Then according to Proposition 2.2, we can check easily that
$$C_{\psi,\varphi}^n=C_{\psi_{(n)},\varphi_n}.$$

For a holomorphic function $\psi$ on $\mathbb{B}_N$, we say $\psi$ is bounded away from zero if $\inf_{z\in \mathbb{B}_N}|\psi(z)|>0$. Note that $\psi$ is bounded away from zero on $\mathbb{B}_N$ if and only if $\frac{1}{\psi}$ is bounded on $\mathbb{B}_N$.

\begin{proposition}
Suppose $\varphi$ is an automorphism of $\mathbb{B}_N$, then $C_{\psi,\varphi}$ is invertible on $H^2(\beta,\mathbb{B}_N)$ if and only if $\psi$ is both bounded and bounded away from zero on $\mathbb{B}_N$.
\end{proposition}

\begin{proof}
Since $\varphi$ is an automorphism of $\mathbb{B}_N$, then the composition operator $C_\varphi$ is invertible on $H^2(\beta,\mathbb{B}_N)$, and $C_\varphi^{-1}=C_{\varphi^{-1}}$. So if $C_{\psi,\varphi}$ is invertible on $H^2(\beta,\mathbb{B}_N)$, then $C_{\psi,\varphi}C_\varphi^{-1}$ is also invertible on $H^2(\beta,\mathbb{B}_N)$. However, $C_{\psi,\varphi}C_\varphi^{-1}=M_\psi$, which is the multiplication induced by $\psi$. Thus the invertibility of $M_\psi$ shows that $\psi$ is both bounded and bounded away from zero on $\mathbb{B}_N$.

On the other hand, if $\psi$ is both bounded and bounded away from zero on $\mathbb{B}_N$, then
$\frac{1}{\psi\comp\varphi^{-1}}$ is bounded on $\mathbb{B}_N$. So $C_{\frac{1}{\psi\comp\varphi^{-1}},\varphi^{-1}}$ is bounded on $H^2(\beta,\mathbb{B}_N)$. Then by Proposition 2.2, we have
$$C_{\psi,\varphi}^{-1}C_{\frac{1}{\psi\comp\varphi^{-1}},\varphi^{-1}}= C_{\frac{1}{\psi\comp\varphi^{-1}},\varphi^{-1}}C_{\psi,\varphi}^{-1}=I.$$
Thus we have
$$C_{\psi,\varphi}^{-1}=C_{\frac{1}{\psi\comp\varphi^{-1}},\varphi^{-1}}.$$
\end{proof}

In most situations we require that $\psi$ is continuous up to the $S_N$. So throughout this paper, we always consider the weighted composition operator $C_{\psi,\varphi}$ such that $\varphi$ is an automorphism of $\mathbb{B}_N$ and $\psi\in A(\mathbb{B}_N)$ is bounded away from zero. Here $A(\mathbb{B}_N)$ denotes the set of functions that are holomorphic on $\mathbb{B}_N$ and continuous up to the boundary $S_N$.

If $\varphi$ is an automorphism of $\mathbb{B}_N$ and $\psi\in A(\mathbb{B}_N)$ is bounded away from zero, we also define
$$\psi_{(k)}=\prod_{j=k}^{0}\frac{1}{\psi\comp\varphi_k}$$
for $k<0$. When $k=0$, we set $\psi_{(0)}=1$.

\begin{remark}
Another useful observation of $C_{\psi,\varphi}$ is that we can write
$$C_{\psi,\varphi}=M_\psi C_\varphi$$
whenever $\varphi$ is an automorphism of $\mathbb{B}_N$ and $\psi\in A(\mathbb{B}_N)$.
\end{remark}

The next proposition is about the adjoint operator of $C_{\psi,\varphi}$.

\begin{proposition}
Suppose $C_{\psi,\varphi}$ is bounded on $H^2(\beta,\mathbb{B}_N)$, then we have
$$C_{\psi,\varphi}^*K_z=\overline{\psi(z)}K_{\varphi(z)}$$
for all $z\in \mathbb{B}_N$.
\end{proposition}

\begin{proof}
For any $f\in H^2(\beta,\mathbb{B}_N)$,
\begin{align*}
\langle f,C_{\psi,\varphi}^*K_z\rangle&=\langle C_{\psi,\varphi}f,K_z\rangle
\\&=\langle \psi\cdot f\comp\varphi,K_z\rangle
\\&=\psi(z)\cdot f(\varphi(z))
\\&=\langle f,\overline{\psi(z)}K_{\varphi(z)}\rangle.
\end{align*}
So
$$C_{\psi,\varphi}^*K_z=\overline{\psi(z)}K_{\varphi(z)}.$$
\end{proof}

\subsection{Some other notations}

Let $z$ be a point in $\mathbb{C}^N$, we will write
$$z=(z^{[1]},z^{[2]},...,z^{[N]}),$$
where $z^{[j]}$ is the $j$-th component of $z$. Also we will let $z'$ denote the last $N-1$ components of $z$, that is, $z=(z^{[1]},z')$ and
$$z'=(z^{[2]},z^{[3]},...,z^{[N]})\in\mathbb{C}^{N-1}.$$
By writing $e_1$ we mean the point $(1,0')$, whose first component is $1$ and other components are $0$.

Fix a $\varphi\in Aut(\mathbb{B}_N)$, we call the sequence of points $\{z_k\}_{k=-\infty}^{+\infty}$ an iteration sequence for $\varphi$ if $z_{k+1}=\varphi(z_k)$ for all $k\in\mathbb{Z}$.

For a operator $T$ on a Hilbert space, we use $\sigma(T)$ denote the spectrum of $T$ and $r(T)$ denote the spectral radius of $T$.

\section{Automorphisms with no fixed point in $\mathbb{B}_N$ and two fixed points on $S_N$}

\subsection{Preparations}

First, we assume temporarily that the Denjoy-Wolff point of $\varphi$ is $e_1=(1,0')$ and the other fixed points of $\varphi$ is $-e_1=(-1,0')$. The next lemma shows that this can benefit us much in simplifying the calculation.
\begin{lemma}
Suppose $\varphi\in Aut(\mathbb{B}_N)$ fixes the points $e_1$ and $-e_1$ only, with Denjoy-Wolff point $e_1$. Then
$$\varphi(z)=\left(\frac{z^{[1]}+s}{1+sz^{[1]}},U\frac{\sqrt{1-s^2}z'}{1+sz^{[1]}}\right)$$
for some $U$ unitary on $\mathbb{C}^{N-1}$ and $s\in(0,1)$.
\end{lemma}

\begin{proof}
Let
$$\mathcal{M}=\left\{(z^{[1]},z')\in \mathbb{B}_N : z'=0\right\}.$$

Note that $\mathcal{M}$ is the only affine subset in $\mathbb{B}_N$ of dimension $1$ that contains $e_1$ and $-e_1$.

Since the image of any affine subset in $\mathbb{B}_N$ under $\varphi$ is also an affine subset in $\mathbb{B}_N$ of the same dimension ( see Theorem 2.74 in \cite{CM} ), $\varphi$ acts as an automorphism when it is restricted on $\mathcal{M}$. So $\varphi^{-1}(0)$ belongs to $\mathcal{M}$. Thus we can write $\varphi^{-1}(0)$ as $a=(\xi,0')$ for some $\xi\in D$. Then
$$\varphi_a(z)=\left(\frac{\xi-z^{[1]}}{1-\overline\xi z^{[1]}},\frac{-\sqrt{1-|\xi|^2}z'}{1-\overline\xi z^{[1]}}\right)$$
gives the involutive automorphism that exchange $0$ and $a$. So $\phi=\varphi\comp\varphi_a$ is an automorphism that fixes the point $0$, hence a unitary map. Moreover, since $\mathcal{M}$ is invariant under $\varphi$ and $\varphi_a$, it is also invariant under $\phi$, which means that
$$\phi(z)=(e^{i\theta}z^{[1]},Vz')$$
for some $\theta\in\mathbb{R}$ and $V$ unitary on $\mathbb{C}^{N-1}$. Therefore
$$\varphi(z)=\phi\comp\varphi_a(z)=\left(e^{i\theta}\frac{\xi-z^{[1]}}{1-\overline\xi z^{[1]}},-V\frac{\sqrt{1-|\xi|^2}z'}{1-\overline\xi z^{[1]}}\right).$$

By using the conditions $\varphi(e_1)=e_1$ and $\varphi(-e_1)=-e_1$, we get the following equations
$$\left\{\begin{array}{l}
   e^{i\theta}(\xi-1)=1-\overline\xi \\
   e^{i\theta}(\xi+1)=-1-\overline\xi. \\
   \end{array}\right.$$
So we have $e^{i\theta}=-1$ and $\xi=\overline\xi\in\mathbb{R}$. By taking $U=-V$ and $s=-\xi$, we get the expected expression of $\varphi$.

Finally, since $e_1$ is the Denjoy-Wolff point of $\varphi$,
$$\frac{\partial\varphi^{[1]}}{\partial z^{[1]}}(e_1)=\frac{1-s}{1+s}<1.$$
Thus we have $s\in(0,1)$.
\end{proof}

\begin{remark}
The lemma above shows that if $\varphi\in Aut(\mathbb{B}_N)$ fixes the points $e_1$ and $-e_1$ only, with Denjoy-Wolff point $e_1$, then
$$\varphi^{[1]}(z)=\frac{z^{[1]}+s}{1+sz^{[1]}}$$
for some $s\in(0,1)$. It is trivial to check that the inverse is also true.
\end{remark}

The next lemma is just a restatement of some familiar results by our notations.

\begin{lemma}
Suppose $\varphi\in Aut(\mathbb{B}_N)$ fixes the points $e_1$ and $-e_1$ only, with Denjoy-Wolff point $e_1$. Then
$$dv(\varphi(z))=\left(\frac{\partial\varphi^{[1]}}{\partial z^{[1]}}(z)\right)^{N+1}dv(z)$$
for all $z$ in $\mathbb{B}_N$ and
$$d\sigma(\varphi(\zeta))=\left(\frac{\partial\varphi^{[1]}}{\partial z^{[1]}}(\zeta)\right)^Nd\sigma(\zeta)$$
for all $\zeta$ on $S_N$.
\end{lemma}

\begin{proof}
By Remark 3.2, we can find $U$ unitary on $\mathbb{C}^{N-1}$ and $s\in(0,1)$ such that
$$\varphi(z)=\left(\frac{z^{[1]}+s}{1+sz^{[1]}},U\frac{\sqrt{1-s^2}z'}{1+sz^{[1]}}\right).$$

Then by a simple computation we get
$$\varphi'(z)=\left(
\begin{array}{ccc}
 \frac{1-s^2}{(1+sz^{[1]})^2} & 0 \\
 0 & \frac{\sqrt{1-s^2}}{1+sz^{[1]}}U
\end{array}
\right).$$
So
$$\det\varphi'(z)=\left(\frac{1-s^2}{(1+sz^{[1]})^2}\right)^{(N+1)/2}=\left(\frac{\partial\varphi^{[1]}}{\partial z^{[1]}}(z)\right)^{(N+1)/2}.$$
Therefore we have
$$dv(\varphi(z))=|\det\varphi'(z)|^2dv(z)=\left(\frac{\partial\varphi^{[1]}}{\partial z^{[1]}}(z)\right)^{N+1}dv(z).$$

Since $\varphi$ is holomorphic in a neighbourhood of $\overline{\mathbb{B}_N}$, the equation above also holds for any $\zeta$ on $S_N$. Thus in polar coordinate we can write
$$dr(\varphi(\zeta))d\sigma(\varphi(\zeta))=\left(\frac{\partial\varphi^{[1]}}{\partial z^{[1]}}(\zeta)\right)^{N+1}dr(\zeta)d\sigma(\zeta).$$
However,
\begin{align*}
\frac{dr(\varphi(\zeta))}{dr(\zeta)}&=\lim_{w\to\zeta}\frac{1-|\varphi(w)|}{1-|w|}
\\&=\lim_{w\to\zeta}\frac{1-|\varphi(w)|^2}{1-|w|^2}
\\&=\lim_{w\to\zeta}\frac{\partial\varphi^{[1]}}{\partial z^{[1]}}(w)=\frac{\partial\varphi^{[1]}}{\partial z^{[1]}}(\zeta).
\end{align*}
So
$$d\sigma(\varphi(\zeta))=\left(\frac{\partial\varphi^{[1]}}{\partial z^{[1]}}(\zeta)\right)^Nd\sigma(\zeta).$$
\end{proof}

In order to generalize our discussion to the cases when the fixed points of $\varphi$ are not $\pm e_1$, the next lemma is also needed.

\begin{lemma}
Suppose $\varphi\in Aut(\mathbb{B}_N)$, with no fixed point in $\mathbb{B}_N$, has two distinct fixed points $a\ne b$ on $S_N$ and $a$ is the Denjoy-Wolff point of $\varphi$. Then we can find $\phi\in Aut(\mathbb{B}_N)$ such that
$$\phi\comp\varphi\comp\phi^{-1}(z)=\left(\frac{z^{[1]}+s}{1+sz^{[1]}},U\frac{\sqrt{1-s^2}z'}{1+sz^{[1]}}\right)$$
for some $U$ unitary on $\mathbb{C}^{N-1}$ and $s\in(0,1)$.

Moreover, the spectral radius of the matrixs $\varphi'(a)^{-1}$ and $\varphi'(b)$ are both $\frac{1+s}{1-s}$
\end{lemma}

\begin{proof}
According to Corollary 2.2.5 in \cite{Aba}, we can find $\phi\in Aut(\mathbb{B}_N)$ such that $\phi(a)=e_1$ and  $\phi(b)=-e_1$. Then $\widehat{\varphi}=\phi\comp\varphi\comp\phi^{-1}$ is an automorphism of $\mathbb{B}_N$ that fixes $\pm e_1$ only, and the Denjoy-Wolff point of $\widehat{\varphi}$ is $e_1$. Thus by Lemma 2.1 we can write
$$\widehat{\varphi}(z)=\left(\frac{z^{[1]}+s}{1+sz^{[1]}},U\frac{\sqrt{1-s^2}z'}{1+sz^{[1]}}\right)$$
for some $U$ unitary on $\mathbb{C}^{N-1}$ and $s\in (0,1)$.

Since both $\varphi$ and $\phi$ are holomorphic in a neighbourhood of $\overline{\mathbb{B}_N}$, we have
$$\widehat{\varphi}'(e_1)=\phi'(a)\cdot\varphi'(a)\cdot(\phi^{-1})'(e_1)=\phi'(a)\cdot\varphi'(a)\cdot\phi'(a)^{-1}.$$
So $\varphi'(a)$, as a matrix, is similar to $\widehat{\varphi}'(e_1)$. Likewise, $\varphi'(b)$ is similar to $\widehat{\varphi}'(-e_1)$.

By a simple computation, we get
$$\widehat{\varphi}'(e_1)=\left(
\begin{array}{ccc}
 \frac{1-s}{1+s} & 0 \\
 0 & \sqrt{\frac{1-s}{1+s}}U
\end{array}
\right)$$
and
$$\widehat{\varphi}'(-e_1)=\left(
\begin{array}{ccc}
 \frac{1+s}{1-s} & 0 \\
 0 & \sqrt{\frac{1+s}{1-s}}U
\end{array}
\right).$$

Since $U$ is unitary, $\frac{1+s}{1-s}$ is the spectral radius of $\widehat{\varphi}'(-e_1)$ and $\widehat{\varphi}'(e_1)^{-1}$, hence by the similarity, is also the spectral radius of $\varphi'(b)$ and $\varphi'(a)^{-1}$.
\end{proof}

\subsection{Spectral radius}

In this section, we find out the spectral radius of $C_{\psi,\varphi}$ by, of course, estimating the norm of $C_{\psi,\varphi}^n=C_{\psi_{(n)},\varphi_n}$. For this purpose, we will turn to the Siegel upper half space. Recall that the Siegel upper half space $H_N$ is defined by
$$H_N=\{w\in\mathbb{C}^N : \text{Im } w^{[1]}>|w'|^2\},$$
and the Cayley transform $\Phi$ given by
$$\Phi(z)=i\frac{e_1+z}{1-z^{[1]}}$$
is a biholomorphism between $\mathbb{B}_N$ and $H_N$. Moreover, $\Phi$ extends to a homeomorphism of $\overline{\mathbb{B}_N}$ onto $H_N\cup\partial H_N\cup\{\infty\}$.

If $\varphi$ fixes the points $e_1$ and $-e_1$ only, then $\sigma=\Phi\comp\varphi\comp\Phi^{-1}$ is a automorphism of $H_N$ that fixs $0$ and $\infty$ only. By Proposition 2.2.11 in \cite{Aba}, we can write
$$\sigma(w)=(r^2w^{[1]},rUw')$$
where $U$ is unitary on $\mathbb{C}^{N-1}$ and $0<r<1$.

Suppose $\psi\in A(\mathbb{B}_N)$ is bounded away from zero, then for any $\epsilon>0$ we can find neighborhoods $V_1$ and $V_2$ of $e_1$ and $-e_1$ respectively such that
$$|\psi(z)|\leqslant (1+\epsilon)\text{max}\{|\psi(e_1)|,|\psi(-e_1)|\}$$
for all $z\in(V_1\cup V_2)\cap S_N$.

Since $\Phi(e_1)=0$ and $\Phi(-e_1)=\infty$, we can find $R>1$ such than $\{w\in\partial H_N : |w|<\frac{1}{R}\}\subset \Phi(V_1\cap S_N)$ and $\{w\in\partial H_N : |w|>R\}\subset \Phi(V_2\cap S_N)$. By noticing that$|\sigma(w)|\leqslant r|w|$, we can find $n_0\in\mathbb{N}^*$ such that
$$\sigma_n(w)\in\{w\in\partial H_N : |w|<\frac{1}{R}\}\subset\Phi(V_1\cap S_N)$$
whenever $n>n_0$ and
$$w\in\Phi \left(S_N\backslash(V_1\cup V_2)\right)\subset\{w\in\partial H_N : \frac{1}{R}\leqslant|w|\leqslant R\},$$
which means that $\varphi_n(z)\in V_1\cap S_N$ whenever $z\in S_N\backslash(V_1\cup V_2)$. So for all $z\in S_N$, at most $n_0$ elements of the set $\{\varphi_j(z) : j=0,1,2,...\}$ are not contained in $(V_1\cup V_2)\cap S_N$. Thus for $n>n_0$,
\begin{align*}
||\psi_{(n)}||_{\infty}&=\max_{|z|=1}\left(\prod_{j=0}^{n-1}|\psi\comp\varphi_j(z)|\right) \\&\leqslant||\psi||_{\infty}^{n_0}\left[(1+\epsilon)\text{max}\{|\psi(e_1)|,|\psi(-e_1)|\}\right]^{n-n_0}.
\end{align*}
By the arbitrariness of $\epsilon$,
$$\overline{\lim_{n\to\infty}}||\psi_{(n)}||_{\infty}^{1/n}\leqslant\text{max}\{|\psi(e_1)|,|\psi(-e_1)|\}.$$

However, since $\psi_{(n)}(e_1)=\psi(e_1)^n$ and $\psi_{(n)}(-e_1)=\psi(-e_1)^n$, we have
$$||\psi_{(n)}||_{\infty}^{1/n}\geqslant\text{max}\{|\psi(e_1)|,|\psi(-e_1)|\}.$$
Thus we can get
$$\lim_{n\to\infty}||\psi_{(n)}||_{\infty}^{1/n}=\text{max}\{|\psi(e_1)|,|\psi(-e_1)|\}.$$

This limitation plays a critical role in our next lemma.

\begin{lemma}
Suppose $\varphi\in Aut(\mathbb{B}_N)$ where
$$\varphi^{[1]}(z)=\frac{z^{[1]}+s}{1+sz^{[1]}}\quad,\quad 0<s<1$$
and $\psi\in A(\mathbb{B}_N)$ is bounded away from zero. Then the spectral radius of $C_{\psi,\varphi}$ on $H^2(\beta,\mathbb{B}_N)$, for $\beta(n)=(n+1)^{1-\gamma}$ with $\gamma\geqslant1$, is no larger than $$\max\left\{|\psi(e_1)|\left(\frac{1+s}{1-s}\right)^K,|\psi(-e_1)|\left(\frac{1-s}{1+s}\right)^K\right\},$$
where $K=(N-1+\gamma)/2$.
\end{lemma}

\begin{proof}
For any fixed $n\in\mathbb{N}^*$,
\begin{align*}
\psi_{(n)}=\prod_{j=0}^{n-1}\psi\comp\varphi_j&=\left(\prod_{j=0}^{n-1}\frac{\psi\comp\varphi_j}{\left(\frac{\partial\varphi^{[1]}}{\partial z^{[1]}}\comp\varphi_j\right)^K}\right)\left(\prod_{j=0}^{n-1}\frac{\partial\varphi^{[1]}}{\partial z^{[1]}}\comp\varphi_j\right)^K\\&=\left(\prod_{j=0}^{n-1}\frac{\psi}{\left(\partial\varphi^{[1]}/\partial z^{[1]}\right)^K}\comp\varphi_j\right)\left(\frac{\partial\varphi_n^{[1]}}{\partial z^{[1]}}\right)^K.
\end{align*}

Let
$$\upsilon=\frac{\psi}{\left(\partial\varphi^{[1]}/\partial z^{[1]}\right)^K}\quad,\quad \nu=\left(\frac{\partial\varphi_n^{[1]}}{\partial z^{[1]}}\right)^K,$$
then $\psi_{(n)}=\upsilon_{(n)}\cdot\nu$. So
\begin{align*}
||C_{\psi,\varphi}^n||^{1/n}=||C_{\psi_{(n)},\varphi_n}||^{1/n}&=||M_{\upsilon_{(n)}}C_{\nu,\varphi_n}||^{1/n}
\\&\leqslant||\upsilon_{(n)}||_\infty^{1/n}\cdot||C_{\nu,\varphi_n}||^{1/n}
\end{align*}

Since $\upsilon\in A(\mathbb{B}_N)$ is bounded away from zero, the argument before the theorem shows that
\begin{align*}
\lim_{n\to\infty}||\upsilon_{(n)}||_{\infty}^{1/n}&=\max\left\{|\upsilon(e_1)|,|\upsilon(-e_1)|\right\}\\ &=\max\left\{|\psi(e_1)|\left(\frac{1+s}{1-s}\right)^K,|\psi(-e_1)|\left(\frac{1-s}{1+s}\right)^K\right\}.
\end{align*}

Now we estimate the norm of $C_{\nu,\varphi_n}$, we treat the situations whether $\gamma=1$ or not separately.

When $\gamma=1$, the space is $H^2(\mathbb{B}_N)$. So for any $f$ in the space,
\begin{align*}
||C_{\nu,\varphi_n}f||^2&=\int_{S_N}|f\comp\varphi_n(z)|^2|\nu(z)|^2d\sigma(z) \\&=\int_{S_N}|f\comp\varphi_n(z)|^2\left|\frac{\partial\varphi_n^{[1]}}{\partial z^{[1]}}(z)\right|^Nd\sigma(z).
\end{align*}
Since $\varphi_n$ also fixes $\pm e_1$ only, by taking $w=\varphi_n(z)$, Lemma 3.3 shows that
$$||C_{\nu,\varphi_n}f||^2=\int_{S_N}|f(w)|^2d\sigma(w)=||f||^2.$$
Thus we have $||C_{\nu,\varphi_n}||=1$.

When $\gamma>1$, $H^2(\beta,\mathbb{B}_N)$ is equipped with the norm on the weighted Bergman space $A^2_{\gamma-2}(\mathbb{B}_N)$. So for any $f\in H^2(\beta,\mathbb{B}_N)$,
\begin{align*}
||C_{\nu,\varphi_n}f||^2&=\int_{\mathbb{B}_N}c_{\gamma-2}|f\comp\varphi_n(z)|^2|\nu(z)|^2(1-|z|^2)^{\gamma-2}dv(z) \\&=\int_{\mathbb{B}_N}c_{\gamma-2}|f\comp\varphi_n(z)|^2(1-|z|^2)^{\gamma-2}\left|\frac{\partial\varphi_n^{[1]}}{\partial z^{[1]}}(z)\right|^{N-1+\gamma}dv(z)
\\&=\int_{\mathbb{B}_N}c_{\gamma-2}|f\comp\varphi_n(z)|^2(1-|\varphi_n(z)|^2)^{\gamma-2}\left|\frac{\partial\varphi_n^{[1]}}{\partial z^{[1]}}(z)\right|^{N+1}dv(z).
\end{align*}
The last equality above follows from the fact that
$$\left|\frac{\partial\varphi_n^{[1]}}{\partial z^{[1]}}(z)\right|=\frac{1-|\varphi_n(z)|^2}{1-|z|^2}.$$
Again by Lemma 3.3, taking $w=\varphi_n(z)$ we get
$$||C_{\nu,\varphi_n}f||^2=\int_{\mathbb{B}_N}c_{\gamma-2}|f(w)|^2(1-|w|^2)^{\gamma-2}dv(w)=||f||^2.$$

Thus for any $n\in\mathbb{N}^*$ and $\gamma\geqslant1$, the norm of $C_{\nu,\varphi_n}$ on $H^2(\beta,\mathbb{B}_N)$ is $1$. So
\begin{align*}
r(C_{\psi,\varphi})&=\lim_{n\to\infty}||C_{\psi,\varphi}^n||^{1/n}
\\&\leqslant\lim_{n\to\infty}||\upsilon_{(n)}||_{\infty}^{1/n}\cdot||C_{\nu,\varphi_n}||^{1/n}
\\&=\max\left\{|\psi(e_1)|\left(\frac{1+s}{1-s}\right)^K,|\psi(-e_1)|\left(\frac{1-s}{1+s}\right)^K\right\}.
\end{align*}

\end{proof}

\begin{corollary}
Suppose $\varphi\in Aut(\mathbb{B}_N)$ where
$$\varphi^{[1]}(z)=\frac{z^{[1]}+s}{1+sz^{[1]}}\quad,\quad 0<s<1$$
and $\psi\in A(\mathbb{B}_N)$ is bounded away from zero. Then the spectrum of $C_{\psi,\varphi}$ on $H^2(\beta,\mathbb{B}_N)$, for $\beta(n)=(n+1)^{1-\gamma}$ with $\gamma\geqslant1$, is contained in the annulus
$$\left\{\lambda : \min\left\{|\psi(e_1)|\rho^K,|\psi(-e_1)|\rho^{-K}\right\} \leqslant\lambda\leqslant \max\left\{|\psi(e_1)|\rho^K,|\psi(-e_1)|\rho^{-K}\right\}\right\}$$
where $\rho=\frac{1+s}{1-s}$ and $K=(N-1+\gamma)/2$.
\end{corollary}

\begin{proof}
By Lemma 3.5, the spectrum of $C_{\psi,\varphi}$ is contained in the disk
$$\left\{\lambda : |\lambda|\leqslant\max\left\{|\psi(e_1)|\rho^K,|\psi(-e_1)|\rho^{-K}\right\}\right\}.$$

In order to determine the inner radius, let
$$\widehat{\varphi}=\phi\comp\varphi^{-1}\comp\phi^{-1}\quad, \quad\widehat{\psi}=\frac{1}{\psi\comp\varphi^{-1}}\comp\phi^{-1},$$
where $\phi(z)=-z$. Then it is easy to check that
$$\widehat{\varphi}^{[1]}(z)=\varphi^{[1]}(z)=\frac{z^{[1]}+s}{1+sz^{[1]}}$$
and $\Tilde{\psi}\in A(\mathbb{B}_N)$ is also bounded away from zero. Moreover, since $$C_{\psi,\varphi}^{-1}=C_{\frac{1}{\psi\comp\varphi^{-1}},\varphi^{-1}},$$
for any $f\in H^2(\beta,\mathbb{B}_N)$,
$$C_{\phi}^{-1}C_{\psi,\varphi}^{-1}C_{\phi}(f)=\widehat{\psi}\cdot f\comp\widehat{\varphi}=C_{\widehat{\psi},\widehat{\varphi}}(f).$$

Thus $C_{\psi,\varphi}^{-1}$ is similar to $C_{\widehat{\psi},\widehat{\varphi}}$. So the spectrum of $C_{\psi,\varphi}^{-1}$, which is also the spectrum of $C_{\widehat{\psi},\widehat{\varphi}}$, is contained in the disk
\begin{align*}
&\left\{\lambda : |\lambda|\leqslant\max\left\{|\widehat{\psi}(e_1)|\rho^K,|\widehat{\psi}(-e_1)|\rho^{-K}\right\}\right\}
\\=&\left\{\lambda : |\lambda|\leqslant\max\left\{\frac{1}{|\psi(-e_1)|}\rho^K,\frac{1}{|\psi(e_1)|}\rho^{-K}\right\}\right\}.
\end{align*}
This means that the spectrum of $C_{\psi,\varphi}$ is contained in the set
$$\left\{\lambda : |\lambda|\geqslant\min\left\{|\psi(e_1)|\rho^K,|\psi(-e_1)|\rho^{-K}\right\}\right\}.$$
\end{proof}

\subsection{The spectra}

Next lemma shows that the composition operators we consider here are circularly symmetric.

\begin{lemma}
Suppose $\varphi\in Aut(\mathbb{B}_N)$ where
$$\varphi^{[1]}(z)=\frac{z^{[1]}+s}{1+sz^{[1]}}\quad,\quad 0<s<1$$
and $\psi\in A(\mathbb{B}_N)$ is bounded away from zero. If $\lambda\in\sigma(C_{\psi,\varphi})$, then for any $\theta\in\mathbb{R}$, $e^{i\theta}\lambda\in\sigma(C_{\psi,\varphi})$.
\end{lemma}

\begin{proof}
By the proof of Theorem 7.21 in \cite{CM}, for any $\theta\in\mathbb{R}$ we can found $f\in H^\infty(D)$ bounded away from zero such that
$$f\comp\varphi^{[1]}(z^{[1]},0')=e^{i\theta}f(z^{[1]})$$
for all $z^{[1]}\in D$.

Now let
$$F(z^{[1]},z')=f(z^{[1]}),$$
then $F\in H^\infty(\mathbb{B}_N)$ and $F$ is also bounded away from zero. Moreover, we have that
\begin{align*}
F\comp\varphi(z^{[1]},z')&=f\comp\varphi^{[1]}(z^{[1]},z')\\&=f\comp\varphi^{[1]}(z^{[1]},0') \\&=e^{i\theta}f(z^{[1]})\\&=e^{i\theta}F(z^{[1]},z').
\end{align*}
So for any $g\in H^2(\mathbb{B}_N)$,
\begin{align*}
M_F^{-1}C_{\psi,\varphi}M_F(g)&=M_F^{-1}M_\psi C_\varphi M_F(g)\\&=M_\psi M_F^{-1}((F\comp\varphi)\cdot(g\comp\varphi))\\&=e^{i\theta}M_\psi(g\comp\varphi)\\&=e^{i\theta}C_{\psi,\varphi}(g).
\end{align*}
Therefore $C_{\psi,\varphi}$ is similar to $e^{i\theta}C_{\psi,\varphi}$. Thus by the spectral mapping theorem and the fact that similar operators have the same spectrum, we come to the conclusion.
\end{proof}

\begin{corollary}
Suppose $\varphi\in Aut(\mathbb{B}_N)$ where
$$\varphi^{[1]}(z)=\frac{z^{[1]}+s}{1+sz^{[1]}}\quad,\quad 0<s<1$$
and $\psi\in A(\mathbb{B}_N)$ is bounded away from zero. Let $\beta(n)=(n+1)^{1-\gamma}$ with $\gamma\geqslant1$. If
$$|\psi(-e_1)|\left(\frac{1-s}{1+s}\right)^K=|\psi(e_1)|\left(\frac{1+s}{1-s}\right)^K,$$
where $K=(N-1+\gamma)/2$, then the spectrum of $C_{\psi,\varphi}$ on $H^2(\beta,\mathbb{B}_N)$ is the circle
$$\left\{\lambda : |\lambda|=|\psi(-e_1)|\left(\frac{1-s}{1+s}\right)^K=|\psi(e_1)|\left(\frac{1+s}{1-s}\right)^K\right\}.$$
\end{corollary}

\begin{proof}
Since the spectrum of $C_{\psi,\varphi}$ can not be empty, the conclusion follows directly from Corollary 3.6 and Lemma 3.7.
\end{proof}

\begin{lemma}
Suppose $\varphi\in Aut(\mathbb{B}_N)$ where
$$\varphi^{[1]}(z)=\frac{z^{[1]}+s}{1+sz^{[1]}}\quad,\quad 0<s<1$$
and $\psi\in A(\mathbb{B}_N)$ is bounded away from zero. Let $\beta(n)=(n+1)^{1-\gamma}$ with $\gamma\geqslant1$. If
$$|\psi(-e_1)|\left(\frac{1-s}{1+s}\right)^K\leqslant|\psi(e_1)|\left(\frac{1+s}{1-s}\right)^K$$
where $K=(N-1+\gamma)/2$, then the spectrum of $C_{\psi,\varphi}$ on $H^2(\beta,\mathbb{B}_N)$ is
$$\left\{\lambda : |\psi(-e_1)|\left(\frac{1-s}{1+s}\right)^K\leqslant|\lambda| \leqslant |\psi(e_1)|\left(\frac{1+s}{1-s}\right)^K\right\}$$

Moreover, each interior point of the annulus belongs to the point spectrum of  $C_{\psi,\varphi}$.
\end{lemma}

\begin{proof}
Suppose $\lambda\in\mathbb{C}$ satisfies
$$|\psi(-e_1)|\left(\frac{1-s}{1+s}\right)^K<\lambda<|\psi(e_1)|\left(\frac{1+s}{1-s}\right)^K,$$
then we can take $R_1$ and $R_2$ such that
$$|\psi(-e_1)|\left(\frac{1-s}{1+s}\right)^K<R_1<\lambda<R_2<|\psi(e_1)|\left(\frac{1+s}{1-s}\right)^K.$$
Since $\frac{1-s}{1+s}<1$, we can find $p>0$ such that
$$|\psi(e_1)|\left(\frac{1+s}{1-s}\right)^{-p}<R_1$$
and
$$|\psi(-e_1)|\left(\frac{1-s}{1+s}\right)^{-p}<R_2.$$
Let $g(z)=(1-z_1^2)^p$.

Now let $z_k=\varphi_k(0)$ for $k\in\mathbb{Z}$. Since
$$\varphi^{[1]}(z)=\frac{z^{[1]}+s}{1+sz^{[1]}},$$
by Theorem 2.6.5 in \cite{CM}, $\{z_k^{[1]}\}_{k=-\infty}^{+\infty}$ is a interpolating sequence for $H^\infty(D)$. So we can find $f\in H^\infty(D)$ such that $|f(z_k^{[1]})|=1$ and
$$\lambda^{-k}\psi_{(k)}(0)\cdot g(z_k)f(z_k^{[1]})>0$$
for all $k\in\mathbb{Z}$. Let
$$F(z)=f(z^{[1]})=0\quad,\quad z\in \mathbb{B}_N,$$
then $F\in H^\infty(\mathbb{B}_N)$. Also $|F(z_k)|=1$ and
$$\lambda^{-k}\psi_{(k)}(0)\cdot g(z_k)F(z_k)>0.$$

Let
\begin{align*}
h&=\sum_{k=-\infty}^{+\infty}\lambda^{-k}\psi_{(k)}\cdot g\comp\varphi_k\cdot F\comp\varphi_k
\\&=gF+\sum_{k=0}^{+\infty}(h^+_k+h^-_k),
\end{align*}
where
$$h^+_k=\lambda^{-k}\left(\prod_{j=0}^{k-1}\psi\comp\varphi_j\right)\cdot g\comp\varphi_k\cdot F\comp\varphi_k$$
and
$$h^-_k=\lambda^k\left(\prod_{j=1}^k\frac{1}{\psi\comp\varphi_{-j}}\right)\cdot g\comp\varphi_{-k}\cdot F\comp\varphi_{-k}.$$
Then we have
$$h(0)=\sum_{k=-\infty}^{+\infty}\lambda^{-k}\psi_{(k)}(0)\cdot g(z_k)F(z_k)>0,$$
so $h\ne 0$. Also it is easy to check that
$$(C_{\psi,\varphi}-\lambda I)h=0.$$
Thus we have $\lambda$ belongs to the point spectrum of  $C_{\psi,\varphi}$ if we can show $h$ converges in  $H^2(\beta,\mathbb{B}_N)$.

For this end, what we should do is almost the same as what have been done in the proof of Theorem 4.9 in \cite{Hyv}. More precisely, by redoing the last part of Theorem 4.9 in \cite{Hyv} on $H^2(\mathbb{B}_N)$ and $A_\alpha(\mathbb{B}_N)$, we can get directly that $\sum_{k=0}^{\infty}||h_k^+||<\infty$. The key point is the usage if Lemma 3.3. On the other hand, by considering $\varphi^{-1}$ and $\frac{1}{\psi\comp\varphi{-1}}$ instead of $\varphi$ and $\psi$, we can also get $\sum_{k=0}^{\infty}||h_k^-||<\infty$. Thus $h$ is well defined. Considering the length of our proof, we omit the details here.
\end{proof}

\begin{lemma}
Suppose $\varphi\in Aut(\mathbb{B}_N)$ where
$$\varphi^{[1]}(z)=\frac{z^{[1]}+s}{1+sz^{[1]}}\quad,\quad 0<s<1$$
and $\psi\in A(\mathbb{B}_N)$ is bounded away from zero. Let $\beta(n)=(n+1)^{1-\gamma}$ with $\gamma\geqslant1$. If
$$|\psi(e_1)|\left(\frac{1+s}{1-s}\right)^K\leqslant|\psi(-e_1)|\left(\frac{1-s}{1+s}\right)^K,$$
where $K=(N-1+\gamma)/2$, then the spectrum of $C_{\psi,\varphi}$ on $H^2(\beta,\mathbb{B}_N)$ is
$$\left\{\lambda : |\psi(e_1)|\left(\frac{1+s}{1-s}\right)^K\leqslant|\lambda| \leqslant|\psi(-e_1)|\left(\frac{1-s}{1+s}\right)^K\right\}.$$

Moreover, each interior point of the annulus belongs to the point spectrum of  $C_{\psi,\varphi}^*$.
\end{lemma}

\begin{proof}
Let $z_k=\varphi_k(0)$ for $k\in\mathbb{Z}$, then
$$\lim_{k\to +\infty}z_k=e_1\quad,\quad\lim_{k\to -\infty}z_k=-e_1$$

Let $g_k=\frac{K_{z_k}}{||K_{z_k}||}$ where $K_{z_k}$ denote the kernel for evaluation at $z_k$ in $H^2(\beta,\mathbb{B}_N)$. Then by Proposition 2.5,
$$C_{\psi,\varphi}^*g_k=u_kg_{k+1},$$
where
$$u_k=\overline{\psi(z_k)}\frac{||K_{z_{k+1}}||}{||K_{z_k}||}= \overline{\psi(z_k)}\left(\frac{1-|z_k|^2}{1-|z_{k+1}|^2}\right)^K.$$

A simple calculation shows that
$$\frac{1-|z_{k+1}|^2}{1-|z_k|^2}=\left|\frac{\partial\varphi^{[1]}}{\partial z^{[1]}}(z_k)\right|,$$
so we have
\begin{align*}
\lim_{k\to +\infty}|u_k|&=|\psi(e_1)|\cdot\lim_{k\to +\infty}\left|\frac{\partial\varphi^{[1]}}{\partial z^{[1]}}(z_k)\right|^{-K}
\\&=|\psi(e_1)|\left|\frac{\partial\varphi^{[1]}}{\partial z^{[1]}}(z_k)\right|^{-K}
\\&=|\psi(e_1)|\left(\frac{1+s}{1-s}\right)^K.
\end{align*}
Also we have
$$\lim_{k\to -\infty}|u_k|=|\psi(-e_1)|\left(\frac{1-s}{1+s}\right)^K.$$

Now suppose $\lambda\in\mathbb{C}$ satisfies
$$|\psi(e_1)|\left(\frac{1+s}{1-s}\right)^K<\lambda<|\psi(-e_1)|\left(\frac{1-s}{1+s}\right)^K$$
and let
$$h=g_0+\sum_{j=1}^{+\infty}\left(a_jg_j+b_jg_{-j}\right),$$
where $$a_j=\lambda^{-j}\prod_{k=0}^{j-1}u_k\quad,\quad b_j=\lambda^j\prod_{k=-j}^{-1}\frac{1}{u_k}.$$

Since $\lim_{k\to \infty}\left|\frac{u_k}{\lambda}\right|<1$, we can find $q\in(0,1)$ and $n_0\in\mathbb{N}^*$ such that
$$|a_n|=|\lambda^{-n}\prod_{k=0}^{n-1}u_k|<q^n$$
when $n>n_0$. Thus we have $\sum_{j=1}^{+\infty}|a_j|<\infty$. Also by $\lim_{k\to \infty}\left|\frac{\lambda}{u_k}\right|<1$, we have $\sum_{j=1}^{+\infty}|b_j|<\infty$. So $h$ is well defined in $H^2(\beta,\mathbb{B}_N)$. It is easy to check that
$$(C_{\psi,\varphi}^*-\lambda I)h=0,$$
therefore we have $\lambda$ belongs to the point spectrum of  $C_{\psi,\varphi}^*$ if we can show $h\ne 0$.

Again by Theorem 2.6.5 in \cite{CM}, $\{z_k^{[1]}\}_{k=-\infty}^{+\infty}$ is a interpolating sequence for $H^\infty(D)$. So we can find $f\in H^\infty(D)$ such that $f(0)=1$ and $f(z_k^{[1]})=0$ when $k\ne 0$. Let
$$F(z)=f(z^{[1]})\quad,\quad z\in \mathbb{B}_N.$$
Then $F\in H^\infty(\mathbb{B}_N)\subset H^2(\beta,\mathbb{B}_N)$. Also $F(0)=1$ and $F(z_k)=0$ when $k\ne 0$. So
$$\langle F,h_\lambda\rangle=\langle F,g_0\rangle=\frac{F(0)}{||K_{z_0}||}=1.$$
Thus we have $h\ne 0$. So $\lambda$ belongs to the point spectrum of  $C_{\psi,\varphi}^*$. Along with Corollary 3.6, we get our conclusion.
\end{proof}

Now we are ready to generalize our result to the cases when the fixed points of $\varphi$ are not necessarily $\pm e_1$.

\begin{theorem}
Suppose $\varphi\in Aut(\mathbb{B}_N)$, with no fixed point in $\mathbb{B}_N$, has two distinct fixed points $a\ne b$ on $S_N$ and $a$ is the Denjoy-Wolff point of $\varphi$. Let $\psi\in A(\mathbb{B}_N)$ be bounded away from zero. Then the spectrum of $C_{\psi,\varphi}$ on $H^2(\beta,\mathbb{B}_N)$, for $\beta(n)=(n+1)^{1-\gamma}$ with $\gamma\geqslant1$, is
$$\left\{\lambda :\min\{|\psi(a)|\rho^K,|\psi(b)|\rho^{-K}\} \leqslant\lambda\leqslant\max\{|\psi(a)|\rho^K,|\psi(b)|\rho^{-K}\}\right\},$$
where $K=(N-1+\gamma)/2$ and $\rho$ is the spectral radius of the matrix $\varphi'(a)^{-1}$ ( also the spectral radius of $\varphi'(b)$ ).

Moreover, each interior point of the annulus belongs to the point spectrum of either $C_{\psi,\varphi}$ or $C_{\psi,\varphi}^*$.
\end{theorem}

\begin{proof}
By lemma 2.4, we can find $\phi\in Aut(\mathbb{B}_N)$ such that
$$\phi\comp\varphi\comp\phi^{-1}(z)=\left(\frac{z^{[1]}+s}{1+sz^{[1]}},U\frac{\sqrt{1-s^2}z'}{1+sz^{[1]}}\right),$$
where $U$ is unitary on $\mathbb{C}^{N-1}$ and $0<s<1$. Also, we have
$$\rho=\frac{1+s}{1-s}.$$

Now let $\widehat{\varphi}=\phi\comp\varphi\comp\phi^{-1}$ and $\widehat{\psi}=\psi\comp\phi^{-1}$. Then for any $f\in H^2(\beta,\mathbb{B}_N)$,
\begin{align*}
C_{\phi}^{-1}C_{\psi,\varphi}C_{\phi}(f)&=C_{\phi}^{-1}C_{\psi,\varphi}(f\comp\phi)
\\&=C_{\phi}^{-1}(\psi\cdot f\comp\phi\comp\varphi)
\\&=\psi\comp\phi^{-1}\cdot f\comp\phi\comp\varphi\comp\phi^{-1}
\\&=\widehat{\psi}\cdot f\comp\widehat{\varphi}=C_{\widehat{\psi},\widehat{\varphi}}.
\end{align*}
So $C_{\psi,\varphi}$ is similar to $C_{\widehat{\psi},\widehat{\varphi}}$, hence
$$\sigma(C_{\psi,\varphi})=\sigma(C_{\widehat{\psi},\widehat{\varphi}}).$$

Since $\widehat{\psi}\in A(\mathbb{B}_N)$ is also bounded away from zero, the results follow directly from Corollary 3.8, Lemma 3.9 and Lemma 3.10.
\end{proof}

Finally, by taking $N=1$, we get the result when $\varphi$ is a hyperbolic automorphism of $D$, which complete the discussion in \cite{Gun2} and \cite{Hyv} on $H^2(D)$ and $A_\alpha^2(D)$, as our corollary below.

\begin{corollary}
Suppose $\varphi$ is a hyperbolic automorphism of $D$ with Denjoy-Wolff point $a$ and the other fixed point $b$. Assume that $\psi\in A(D)$ is bounded away from zero. Then on $H^2(D)$, the spectrum of $C_{\psi,\varphi}$ is
$$\left\{\lambda : \min\{|\psi(a)|\rho^{1/2},|\psi(b)|\rho^{-1/2}\}\leqslant|\lambda|\leqslant \max\{|\psi(a)|\rho^{1/2},|\psi(b)|\rho^{-1/2}\}\right\},$$
and on $A_\alpha^2(D)$, the spectrum of $C_{\psi,\varphi}$ is
$$\left\{\lambda : \min\{|\psi(a)|\rho^{(\alpha+2)/2},|\psi(b)|\rho^{-(\alpha+2)/2}\}\leqslant |\lambda|\leqslant\max\{|\psi(a)|\rho^{(\alpha+2)/2},|\psi(b)|\rho^{-(\alpha+2)/2}\}\right\}.$$
Here $\rho=\varphi'(a)^{-1}=\varphi'(b)$.
\end{corollary}

\section{Automorphisms with no fixed point in $\mathbb{B}_N$ and one fixed points on $S_N$}

\subsection{Spectral radius}

Let us also start by estimating the spectral radius in this situation.

\begin{lemma}
Suppose $\varphi\in Aut(\mathbb{B}_N)$ fixes the point $e_1$ only, and $\psi\in A(\mathbb{B}_N)$ is bounded away from zero. Then the spectral radius of $C_{\psi,\varphi}$ on $H^2(\beta,\mathbb{B}_N)$, for $\beta(n)=(n+1)^{1-\gamma}$ with $\gamma\geqslant1$, is no larger than $|\psi(e_1)|$.
\end{lemma}

\begin{proof}
Let
$$\Phi(z)=i\frac{e_1+z}{1-z^{[1]}}$$
be the Cayley transform from $\mathbb{B}_N$ onto the Siegel upper half space $H_N$. Since $e_1$ is the only fixed point of $\varphi$ in $\overline{\mathbb{B}_N}$, $\sigma=\Phi\comp\varphi\comp\Phi^{-1}$ is a automorphism of $H_N$ that fixes $\infty$ only. Hence by Proposition 2.2.10 in \cite{Aba} and the continuity of $\sigma$ on $\overline{H_N}$, we can find $U$ unitary on $\mathbb{C}^{N-1}$ and $a\in\partial{H_N}$ such that
$$\sigma(w)=(w^{[1]}+a^{[1]}+2i\langle Uw',a'\rangle,Uw'+a')$$
for all $w\in\overline{H_N}$.

By our assumption, $\psi(e_1)\ne 0$. Then for any $\epsilon>0$ we can find a neighborhood $V$ of $e_1$ such that
$$|\psi(z)|\leqslant (1+\epsilon)|\psi(e_1)|$$
for all $z\in S_N\cap V$.

Since $\Phi$, which maps $e_1$ to $\infty$, extends to a homeomorphism of $\overline{\mathbb{B}_N}$ onto $H_N\cup\partial H_N\cup\{\infty\}$, we can take $\Lambda=\Phi(S_N\backslash V)$ and $\Lambda'=\{w+i : w\in\Lambda\}$. Note that $\Lambda$ is a compact set on $\partial H_N$ and $\Lambda'$ is a compact set in $H_N$. Hence we can find $R>0$ such that both $\Lambda$ and $\Lambda'$ are contained in the set $\{w\in\mathbb{C}^N : |w|\leqslant R\}$. Since $e_1$ is the Denjoy-Wolff point of $\varphi$, $\varphi_n$ converges to $e_1$ uniformly on compact subsets of $\mathbb{B}_N$ as $n\to+\infty$. Equivalently, $\sigma_n$ converges to $\infty$ uniformly on compact subsets of $H_N$. Thus we can find a positive integer $n_0$ such that $|\sigma_n(w)|>R+1$ for all $n>n_0$ and $w\in\Lambda'$.

By noticing that $\sigma_n(w+i)=\sigma_n(w)+ie_1$, for any $w$ in $\Lambda$ and $n>n_0$ we have
$$|\sigma_n(w)|\geqslant|\sigma_n(w+i)|-1>R.$$

Now suppose $z\in S_N\backslash V$, then $\Phi(z)\in\Lambda$. So for any $n>n_0$,
$$|\Phi(\varphi_n(z))|=|\sigma_n(\Phi(z))|>R,$$
which means that $\varphi_n(z)$ is no longer in $S_N\backslash V$. So for all $z\in S_N$, at most $n_0$ elements of the set $\{\varphi_j(z) : j=0,1,2,...\}$ are not contained in $S_N\cap V$. Thus for $n>n_0$,
\begin{align*}
||\psi_{(n)}||_{\infty}&=\max_{|z|=1}\left(\prod_{j=0}^{n-1}|\psi\comp\varphi_j(z)|\right) \\&\leqslant||\psi||_{\infty}^{n_0}\left[(1+\epsilon)|\psi(e_1)|\right]^{n-n_0}.
\end{align*}
By the arbitrariness of $\epsilon$,
$$\overline{\lim_{n\to\infty}}||\psi_{(n)}||_{\infty}^{1/n}\leqslant|\psi(e_1)|.$$
However, since $\psi_{(n)}(e_1)=\psi(e_1)^n$, we have
$$||\psi_{(n)}||_{\infty}^{1/n}\geqslant|\psi(e_1)|.$$
Thus we can get
$$\lim_{n\to\infty}||\psi_{(n)}||_{\infty}^{1/n}=|\psi(e_1)|.$$

Proposition 7.11 in \cite{CM} shows that if $\varphi\in Aut(\mathbb{B}_N)$ fixes the point $e_1$ only, then the spectral radius of $C_\psi$ on $H^2(\beta,\mathbb{B}_N)$ is $1$. So
\begin{align*}
r(C_{\psi,\varphi})&=\lim_{n\to\infty}||C_{\psi,\varphi}^n||^{1/n}
\\&\leqslant\lim_{n\to\infty}||\psi_{(n)}||_{\infty}^{1/n}\cdot||C_{\varphi_n}||^{1/n}
\\&=|\psi(e_1)|\cdot r(C_\varphi)=|\psi(e_1)|.
\end{align*}
\end{proof}

\begin{remark}
The proof of Lemma 4.1 shows that if $\varphi\in Aut(\mathbb{B}_N)$ fixes the point $e_1$ only, then $\varphi_n$ converge to $e_1$ uniformly on any compact subset of $\overline{\mathbb{B}_N}\backslash\{e_1\}$ as $n\to+\infty$.
\end{remark}

\begin{corollary}
Suppose $\varphi\in Aut(\mathbb{B}_N)$ fixes the point $e_1$ only, and $\psi\in A(\mathbb{B}_N)$ is bounded away from zero.Then the spectrum of $C_{\psi,\varphi}$ on $H^2(\beta,\mathbb{B}_N)$, for $\beta(n)=(n+1)^{1-\gamma}$ with $\gamma\geqslant1$, is contained in the circle
$$\{\lambda : |\lambda|=|\psi(e_1)|\}.$$
\end{corollary}

\begin{proof}
By Lemma 4.1, the spectrum of $C_{\psi,\varphi}$ on $H^2(\beta,\mathbb{B}_N)$ is contain in
$$\{\lambda : |\lambda|\leqslant|\psi(e_1)|\}.$$
On the other hand, since $\varphi^{-1}$ is also an automorphism of $\mathbb{B}_N$ that fixes $e_1$ only and $\frac{1}{\psi}$ is bounded away from zero on $\mathbb{B}_N$, the spectrum of $C_{\psi,\varphi}^{-1}=C_{\frac{1}{\psi\comp\varphi^{-1}},\varphi^{-1}}$ is contain in
$$\{\lambda : |\lambda|\leqslant|\frac{1}{\psi(e_1)}|\},$$
which means that the spectrum of $C_{\psi,\varphi}$ is also contain in the set
$$\{\lambda : |\lambda|\geqslant|\psi(e_1)|\}.$$
Thus the corollary is proved.
\end{proof}

\subsection{The spectra}

\begin{lemma}
Suppose $\varphi\in Aut(\mathbb{B}_N)$ fixes the point $e_1$ only. Then for any $\delta\in(0,1)$, we can find $z_0\in \mathbb{B}_N$ such that $d(\varphi_{j_1}(z_0),\varphi_{j_2}(z_0))>\delta$ whenever $j_1\ne j_2$. Here $d(\cdot,\cdot)$ denote the pseudo-hyperbolic distance between two points in $\mathbb{B}_N$.
\end{lemma}

\begin{proof}
For any $z\in \mathbb{B}_N$ and $n\in\mathbb{N}^*$,
$$1-d(z,\varphi_n(z))^2=\frac{(1-|z|^2)(1-|\varphi_n(z)|^2)}{|1-\langle z,\varphi_n(z)\rangle|^2}.$$

Let $\mathcal{M}=\{-re_1 : 0\leqslant r\leqslant1\}$. Since
$$\lim_{r\to 1}\varphi_n(-re_1)=\varphi_n(-e_1)\ne-e_1,$$
we have
$$1-\lim_{r\to 1}d(-re_1,\varphi_n(-re_1))^2=\lim_{r\to 1}\frac{(1-|-re_1|^2)(1-|\varphi_n(-re_1)|^2)}{|1-\langle -re_1,\varphi_n(-re_1)\rangle|^2}=0,$$
which means that
$$\lim_{r\to 1}d(-re_1,\varphi_n(-re_1))=1.$$
Thus for each $n\in\mathbb{N}^*$, we can find $r_n\in(0,1)$ such that $d(-re_1,\varphi_n(-re_1))>\delta$ whenever $r\in(r_n,1)$.

Since $\mathcal{M}$ is a compact subset of $\overline{\mathbb{B}_N}\backslash\{e_1\}$, by Remark 4.2, we can find $n_0\in\mathbb{N}^*$ such that $|\varphi_n(-re_1)-e_1|<1-\delta$
for all $n>n_0$ and $r\in(0,1)$. Thus, $|\varphi_n(-re_1)|>\delta$. Also,
$$1-\text{Re}\left(\varphi_n^{[1]}(-re_1)\right)<|\varphi_n(-re_1)-e_1|<1.$$
So $\text{Re}\left(\varphi_n^{[1]}(-re_1)\right)>0$. Therefore,
$$|1-\langle -re_1,\varphi_n(-re_1)\rangle|\geqslant1+r\text{Re}\left(\varphi_n^{[1]}(-re_1)\right)>1.$$
So for any $n>n_0$ and $r\in(0,1)$ we have
\begin{align*}
1-d(-re_1,\varphi_n(-re_1))^2&=\frac{(1-|-re_1|^2)(1-|\varphi_n(-re_1)|^2)}{|1-\langle -re_1,\varphi_n(-re_1)\rangle|^2}
\\&<1-\delta^2.
\end{align*}

Now take $r_0>\max\{r_j : 1\leqslant j\leqslant n_0\}$ and $z_0=-r_0e_1$. Then the discussion above shows that
$$d(z_0,\varphi_n(z_0))>\delta$$
for all $n\in\mathbb{N}^*$. Suppose $j_1\ne j_2$, then
$$d(\varphi_{j_1}(z_0),\varphi_{j_2}(z_0))=d(z_0,\varphi_{|j_1-j_2|}(z_0))>\delta.$$
\end{proof}

The next lemma strengthens the Theorem 7.11 in \cite{CM}.

\begin{lemma}
Suppose $\varphi\in Aut(\mathbb{B}_N)$ fixes the point $e_1$ only. Then we have
$$\lim_{z\to e_1}\frac{1-|\varphi(z)|^2}{1-|z|^2}=1.$$
\end{lemma}

\begin{proof}
Since $e_1$ is the only fixed point of $\varphi$ in $\overline{\mathbb{B}_N}$, the Wolff's lemma in $\mathbb{B}_N$ ( see Theorem 2.2.22 in \cite{Aba} ) shows that
$$\frac{|1-\varphi^{[1]}(z)|^2}{1-|\varphi(z)|^2}\leqslant\frac{|1-z^{[1]}|^2}{1-|z|^2}$$
for all $z\in \mathbb{B}_N$. However, $e_1$ is also the only fixed point of $\varphi^{-1}$ in $\mathbb{B}_N$, so
$$\frac{|1-z^{[1]}|^2}{1-|z|^2}\leqslant\frac{|1-\varphi^{[1]}(z)|^2}{1-|\varphi(z)|^2}.$$
Thus we have
$$\frac{|1-\varphi^{[1]}(z)|^2}{1-|\varphi(z)|^2}=\frac{|1-z^{[1]}|^2}{1-|z|^2}$$
for all $z\in \mathbb{B}_N$.

Let $a=\varphi^{-1}(0)$, then
$$\frac{|1-a^{[1]}|^2}{1-|a|^2}=\frac{|1-\varphi^{[1]}(a)|^2}{1-|\varphi(a)|^2}=1,$$
so
\begin{align*}
\lim_{z\to e_1}\frac{1-|\varphi(z)|^2}{1-|z|^2}&=\lim_{z\to e_1}\frac{1-|a|^2}{|1-\langle a,z\rangle|^2}
\\&=\frac{1-|a|^2}{|1-a^{[1]}|^2}=1.
\end{align*}
\end{proof}

\begin{lemma}
Suppose $\varphi\in Aut(\mathbb{B}_N)$ fixes the point $e_1$ only, and $\psi\in A(\mathbb{B}_N)$ is bounded away from zero. Then the spectrum of $C_{\psi,\varphi}$ on $H^2(\beta,\mathbb{B}_N)$, for $\beta(n)=(n+1)^{1-\gamma}$ with $\gamma\geqslant1$, is the circle
$$\{\lambda : |\lambda|=|\psi(e_1)|\}$$
\end{lemma}

\begin{proof}
By Corollary 4.3, it is sufficient to prove that each point on the circle belongs to the spectrum of $C_{\psi,\varphi}$.

Let $m$ be a fixed positive integer, then by Lemma 4.4, we can find a iteration sequence $\{z_k\}_{k=0}^{+\infty}$ such that $$\left(1-d(z_k,z_j)^2\right)^K<\frac{1}{2e^4m}$$
whenever $k\ne j$. Here $K=(N-1+\gamma)/2$.

Let
$$g_j(z)=\frac{K_{z_j}(z)}{||K_{z_j}||}=\frac{(1-|z_j|^2)^K}{(1-\langle z,w\rangle)^{2K}}.$$
Then for any $k\ne j$,
\begin{align*}
|\langle g_k,g_j\rangle|&=\left(\frac{(1-|z_k|^2)(1-|z_j|^2)}{|1-\langle z_k,z_j\rangle|^2}\right)^K
\\&=\left(1-d(z_k,z_j)^2\right)^K<\frac{1}{2e^4m}.
\end{align*}
Also we have
$$C_{\psi,\varphi}^*g_j=u_jg_{j+1},$$
where
$$u_j=\overline{\psi(z_j)}\left(\frac{1-|z_j|^2}{1-|z_{j+1}|^2}\right)^K.$$
Since $z_j\to e_1$ as $j\to+\infty$, by Lemma 4.5 we have
$$\lim_{j\to +\infty}|u_j|=|\psi(e_1)|\lim_{j\to +\infty}\left(\frac{1-|z_j|^2}{1-|\varphi(z_j)|^2}\right)^K=|\psi(e_1)|.$$

Now suppose that $|\lambda|=|\psi(e_1)|$, then for any $m>0$ we can find $n_0\in\mathbb{N}^*$ such that
$$1-\frac{1}{m}<\left|\frac{u_n}{\lambda}\right|<1+\frac{1}{m}$$
for all $n\geqslant N$.

Let
$$h_m=\sum_{j=0}^{m-1}a_jg_{n_0+j},$$
where $a_0=1$ and $a_j=\lambda^{-j}\prod_{k=0}^{j-1}u_{n_0+j}$ for $j=1,2,...,m-1$. Then
$$|a_j|<(1+\frac{1}{m})^j<(1+\frac{1}{m})^m<e.$$
Also
$$|a_j|>(1-\frac{1}{m})^j>(1-\frac{1}{m})^m>\frac{1}{e}.$$
So
\begin{align*}
||h_m||^2=\langle h_m,h_m\rangle&\geqslant\sum_{j=0}^{m-1}|a_j|^2-\sum_{j_1\ne j_2}|a_{j_1}||a_{j_2}|\cdot|\langle g_{n_0+j_1},g_{n_0+j_2}\rangle|
\\&\geqslant\sum_{j=0}^{m-1}|a_j|^2-\frac{1}{2e^4m}\sum_{j_1\ne j_2}|a_{j_1}||a_{j_2}|
\\&\geqslant\frac{m}{e^2}-\frac{1}{2e^4m}\cdot m^2e^2=\frac{m}{2e^2}.
\end{align*}

However,
$$(C_{\psi,\varphi}^*-\lambda I)h_m=\lambda(a_mg_{n_0+m}-g_{n_0}),$$
where
$$a_m=\lambda^{-m}\prod_{k=0}^{m}u_{n_0+k}.$$
Since $|a_m|<(1+\frac{1}{m})^m<e$, we have
\begin{align*}
||(C_{\psi,\varphi}^*-\lambda I)h_m||^2&\leqslant|\lambda|^2\left(|a_m|^2+1\right)
\\&\leqslant|\psi(e_1)|^2(e^2+1).
\end{align*}

So
$$\frac{||(C_{\psi,\varphi}^*-\lambda I)h_m||^2}{||h_m||^2}\leqslant|\psi(e_1)|^2\cdot\frac{e^2(e^2+1)}{m}.$$
Thus, by letting $m\to+\infty$, we conclude that $\lambda$ belongs to the approximate point spectrum of $C_{\psi,\varphi}^*$.
\end{proof}

Finally we treat the general case when the fixed point of $\varphi$ is not necessarily $e_1$.

\begin{theorem}
Suppose $\varphi\in Aut(\mathbb{B}_N)$ has no fixed point in $\mathbb{B}_N$ and $a$ is the only fixed point of $\varphi$ on $S_N$. Let $\psi\in A(\mathbb{B}_N)$ be bounded away from zero.Then the spectrum of $C_{\psi,\varphi}$ on either $H^2(\mathbb{B}_N)$ or $A^2_\alpha(\mathbb{B}_N)$ is the circle
$$\{\lambda : |\lambda|=|\psi(a)|\}.$$
\end{theorem}

\begin{proof}
Let U be a unitary map that takes $a$ to $e_1$. Then $\widehat{\varphi}=U\comp\varphi\comp U^{-1}$ is an automorphism of $\mathbb{B}_N$ that fixes $e_1$ only. Let $\widehat{\psi}=\psi\comp U^{-1}$, then a same discussion with Theorem 3.11 shows that $C_{\psi,\varphi}$ is similar to $C_{\widehat{\psi},\widehat{\varphi}}$. So
\begin{align*}
\sigma(C_{\psi,\varphi})&=\sigma(C_{\widehat{\psi},\widehat{\varphi}})
\\&=\{\lambda : |\lambda|=|\widehat{\psi}(e_1)|\}
\\&=\{\lambda : |\lambda|=|\psi(a)|\}.
\end{align*}
\end{proof}

\end{document}